\documentclass[11pt]{article}

\usepackage{enumerate,latexsym,amssymb,booktabs,setspace,graphicx,url,float,tocloft}
\usepackage[leqno]{amsmath}
\usepackage[showframe=false]{geometry}
\usepackage{amsthm}
\usepackage{hyperref}
\usepackage{epstopdf}
\usepackage{diagbox}
\usepackage{multirow}
\usepackage[justification=centering]{caption}
\usepackage{enumitem}
\usepackage{color}	
\usepackage{adjustbox}
\usepackage[english]{babel}
\usepackage{booktabs}
\usepackage{tablefootnote}
\usepackage{longtable}
\usepackage[section]{placeins}
\usepackage{tikz}
\usepackage{siunitx}
\usepackage{mathtools}
\usepackage{algorithmic,algorithm}
\usepackage[mathlines]{lineno}
\makeatletter
\newcommand{\leqnomode}{\tagsleft@true}
\newcommand{\reqnomode}{\tagsleft@false}
\makeatother
\usepackage{tikz}
\usepackage{array}
\usepackage{changepage}
\usepackage{mleftright}
\newcolumntype{C}{@{\extracolsep{3cm}}c@{\extracolsep{0pt}}}%
\usetikzlibrary{shapes.geometric, arrows, patterns,fit,decorations.markings,shapes,positioning,calc,intersections}
\tikzstyle{startstop} = [rectangle, minimum width=3cm, minimum height=1cm,text centered, draw=black]
\tikzstyle{decision} = [diamond, minimum width=3cm, minimum height=1cm, text centered, draw=black]
\tikzstyle{arrow} = [thick,->,>=stealth]

\graphicspath{ {./Pictures/} }

\usepackage[nottoc]{tocbibind}
\newtheorem{prop}{Proposition}
\newtheorem{lemma}[prop]{Lemma}

\newtheorem{theor}[prop]{Theorem}

\allowdisplaybreaks[1]
\newcommand{\contentsnameA}{Appendices}
\makeatletter
    \newcommand\tableofcontentsA{%
        \section*{\contentsnameA
            \@mkboth{%
               \MakeUppercase\contentsnameA}{\MakeUppercase\contentsnameA}}%
        \@starttoc{toca}%
        }
\makeatother
\allowdisplaybreaks[1]
\title{Simplification of  inclusion-exclusion on intersections of unions with application to network systems  reliability}
\author{Lukas  Sch{\"a}fer$^{1}$\thanks{Corresponding author at: University of Edinburgh, School of Mathematics, James Clerk Maxwell Building,
The King's Buildings, Peter Guthrie Tait Road, Edinburgh, EH9 3FD, United Kingdom. E-mail address: \href{mailto:lukas.schaefer@ed.ac.uk}{lukas.schaefer@ed.ac.uk}}, Sergio Garc\'{i}a$^{1}$,\\ Vassili Srithammavanh$^2$\\[2mm]
\small{\hspace{0cm}$^1$School of Mathematics, University of Edinburgh, United Kingdom}\\
\small{\hspace{-4.7cm}$^2$AIRBUS Group Innovations, France}}
\date{}
\providecommand{\keywords}[1]{\textbf{\textit{Keywords: }} #1}
\begin{document}
\maketitle

\begin{abstract}
Reliability of safety-critical systems  is a  paramount issue in system engineering because  in most practical situations the reliability of a non series-parallel network system  has to be calculated. Some methods for calculating reliability use the probability principle of inclusion-exclusion. When dealing with complex networks, this leads to very long mathematical expressions which are usually computationally very expensive to calculate. In this paper, we provide a new expression to simplify the probability principle of inclusion-exclusion formula  for intersections of unions, which appear when calculating reliability on non series-parallel network systems. This new expression exploits the presence of many repeated events and has many fewer terms, which significantly reduces  the computational cost. We also show that the general form of the  probability principle of inclusion-exclusion formula  has a double exponential complexity whereas the simplified form has only an exponential complexity  with a linear exponent. Finally, we illustrate how to use this result when calculating the reliability of a door management system in aircraft engineering. 

\end{abstract}
\small{
 }

\section{Introduction}
\reqnomode
 Reliability of a network system is the probability of the system not failing. It is a critical issue in different fields such as computer networks, information networks or gas networks. In particular, reliability of safety-critical network systems \cite{raus14} is an important topic in system engineering. For example,  in aircraft architecture with safety-critical network systems such as fly-by-wire, actuation, fire warning  and door management systems. In most practical situations, the reliability of a complex network system  (e.g. a system that is not series-parallel) has to be exactly calculated \cite{Hwang1981}. There are several methods to calculate or simulate the reliability of a complex system which have been developed  in recent decades. Some classical static modelling techniques, including reliability block diagram models (\cite{Lev07}), fault tree models, and binary decision diagram models, have been widely used to model static systems. A general introduction to these methods can be found in \cite{raus14}. For time-dependent systems,  modeling techniques such as Markov models \cite{xia12}, dynamic fault tree models \cite{bo92} and Petri net models \cite{Zhong2006} have been used.  In this paper, we propose a new  method to calculate the reliability  for static systems with a new way of writing the classical probability principle of inclusion-exclusion formula. The classical probability principle of inclusion-exclusion formula is 
 
\begin{align}\label{eq:prob_ie}
P(\bigcup_{i=1}^n A_i)=\sum_{i=1}^n\left((-1)^{i+1}\sum_{J\subseteq\{1,\ldots,n\},\atop |J|=i}P\left(\bigcap_{j\in J}A_j\right)\right).
\end{align}
The new method detects which combination of events lead to the same event when simplified and has therefore  many fewer summands than the classical formula for  intersections of unions.

Practical reliability calculations often involve very long expressions when the  probability principle of inclusion-exclusion formula (\ref{eq:prob_ie}) is used. Therefore, there are many approaches in in the literature on general network reliability calculations to simplify the probability principle of inclusion-exclusion as, for example, partitioning techniques \cite{Dot79} and  sum of disjoint products (\cite{Bal03}, \cite{2Rai95}, \cite{Rauzy03}, \cite{Yuan88}). This latter technique is the most often used approach with recent results in \cite{Yeh15}, \cite{Yeh07},  \cite{xin12} or \cite{singh16}. In this paper we propose a new approach to simplify the probability principle of inclusion-exclusion and to apply it to the calculation of the reliability of complex  network systems in system engineering. 
	
In system engineering, most network systems have multiple functions that have to be performed  and they are not always independent (e.g., they share components). Furthermore,  most functions in a system have to be redundant. A function is redundant if two disjoint sets of components of the system can perform that function. Reliability can also be increased if different sets of components in the network can perform the same function.   Therefore, these functions are implemented multiple times in the network system through different sets of components and the calculation of the reliability of the network system becomes a very complex task. 

In our paper, we assume that all failure probabilities of the components are known exactly. We do not consider the case when these probabilities are known only approximately (e.g., either by estimation or a confidence interval). If the  different components of the network are independent of each other, then we can easily calculate the reliability  of a set of components. Through this we can calculate the reliability of  one implementation of a function, which is defined as  the probability of the event that one implementation of the function does not fail. Finally, the probability of an intersection of such events can be calculated  easily. But if full independence cannot be assumed, then the calculation becomes very expensive, usually prohibitive.

 When dealing with optimization problems with reliability constraints, which is the motivation of our research, this calculation cost causes that even models for very small networks are completely intractable. The reason is the large number of variables and non-linear constraints involved in the reliability calculation within the optimization model.  Furthermore the approaches to simplify the  probability principle of inclusion-exclusion formula  mentioned  before (e.g.,  sum of disjoint products) are not suitable to be used within an optimization formulation. Also, approximation, lower or upper bounds for the principal of inclusion-exclusion are not suitable for exact optimization  because they are not in general monotone increasing or decreasing with the reliability. We provide an example for the lower bound in  Appendix~\ref{asec:lbce}. Therefore, in practice, optimization models that involve reliability are usually either solved through heuristics or   series-parallel systems are assumed  (\cite{Coit96}, \cite{kuo01}, \cite{tw13}). 

In this paper, we show how to calculate the reliability of a whole network system in which components are not necessarily independent in a way that requires  considerably fewer calculations (and, thus, it is much cheaper computationally) than the direct use of the probability  inclusion-exclusion principle. The key is to exploit the fact that, when dealing with a network system,  the probability  inclusion-exclusion principle has many repeated terms when applied to intersections of unions. 

The rest of the paper is organized as follows. In Section~\ref{sec:mr} we show why it can be expensive to calculate the reliability of a non series-parallel system. We also state the main result (Proposition~\ref{prop1})  which provides a formula to calculate the reliability for a network system in an exact way with  a much lower number of calculations. In Section~\ref{sec:math} we first state an auxiliary result needed to prove Proposition~\ref{prop1} and then we prove the mentioned proposition. Time complexity is analysed in Section 4. Section~\ref{sec:example} shows an example to illustrate how to calculate reliability using the results of this paper. Finally, some conclusions and future perspectives are discussed in Section~\ref{sec:conclusion}.

\section{Motivation and main result}\label{sec:mr}
We start by showing that, if independence cannot be assumed, then  it can be very expensive to calculate the probability of  a non series-parallel network system with multiple functions and implementations. Afterwards, we introduce a result (Proposition~\ref{prop1}) that reduces the number of calculations involved. 
Let $F_i,\ i\in\{1,\ldots,n\}$, be the event that function $i$ of a system does not fail in a specific period of time and  $F_{ij},\  j\in\{1,\ldots,t_i\},$ be the event that implementation $j$ of function $i$ does not fail in a specific period of time.  Let $\mathcal F=\{F_1,\ldots, F_n\}$ be the set of all functions and $\mathcal F_i=\{F_{i1},\ldots,F_{it_i}\}$ be the set of all implementations of function $i$.  Furthermore, let $R$ be the event that the system does not fail. The reliability of the system, $P(R)$, is the probability that no function in $\mathcal F$ fails. A function $F\in\mathcal F$ does not fail if at least one of its implementations does not fail. Therefore,

\begin{align*}
P(R)=P\left(\bigcap_{i=1}^n F_i \right)=P\left(\bigcap_{i=1}^n\left(\bigcup_{j=1}^{t_i} F_{ij}\right)\right).
\end{align*}
Because the different functions and implementations may not  be independent, $P(R)$ is not easily calculable. In order to work on this expression, first we need to establish some notation. Let 

\begin{align*}
W&=\{1,\ldots,t_1\}\times\hdots\times\{1,\hdots,t_n\},\text{ and }\\
 B_w&=\bigcap_{i=1}^n F_{iw_i}\ \text{ for } w=(w_1,\ldots,w_n)\in W.
 \end{align*}
We then have that 

\begin{align*}
P(R)&=P\left(\bigcap_{i=1}^n\left(\bigcup_{j=1}^{t_i} F_{ij}\right)\right)= P\left(\bigcup_{w\in W}\left(\bigcap_{i=1}^n F_{iw_i}\right)\right)= P\left(\bigcup_{w\in W}B_w\right).
\end{align*}
Now the probability principle of inclusion-exclusion can be used and it follows that 

\begin{align}\label{eq:ie}
P(R)=\sum_{t=1}^{|W|}\left(\left(-1\right)^{t+1}\sum_{I\subseteq W,\atop |I|=t}P\left(\bigcap_{j\in I} B_j\right)\right).
\end{align}
The number of summands  in (\ref{eq:ie}), which is equal to the number of possible intersections of $B_w$'s,  is $\sum_{t=1}^{|W|}\binom{|W|}{t}= 2^{|W|}-1$ with $|W|=\prod_{i=1}^n|\mathcal F_i|$. Therefore, we have  double exponential computational complexity. Table~\ref{t:oldcal} shows the number of summands for different values on the number of functions and implementations with the assumption that every function has the same number of implementations. 
\begin{table}[H]
\centering
\begin{tabular}{c|c | r}
$|\mathcal F|$ &  $|\mathcal F_i|$ & Summands\\
\midrule
  2 & 2&15 \\
    2 & 3&5.11$\times 10^2$	  \\
    2 & 4&6.55$\times 10^4$	  \\
   3 &2 & 2.55$\times 10^2$\\
   3& 3 &1.34$\times 10^7$ \\
   3& 4 & 1.84$\times 10^{17}$\\
   4 & 2 & 6.55$\times 10^4$\\
   4 & 3 & 2.41$\times 10^{24}$ \\
   5 & 2 & 4.29$\times 10^9$\\
   5 & 3 & 1.41$\times 10^{73}$
\end{tabular}
\caption{Number of summands in the probability principle \\of inclusion-exclusion formula.}
\label{t:oldcal}
 \end{table}
As can be seen, even  for a small  number of functions and implementations, the calculation of $P(R)$ becomes very expensive. However, note that there are many (a priori different) terms that, when the intersection of the sets is calculated, lead to the same intersection set, that is, 

 \begin{align}\label{def:repterms}
  \exists I,J\subseteq W:\ I\ne J\ \land\ \bigcap_{i\in I}B_i=\bigcap_{j\in J}B_j.
  \end{align}
 For example, let  $\mathcal F=\{F_1,F_2\}$, $\mathcal F_1=\{F_{11},F_{12}\}$ and $\mathcal F_2=\{F_{21},F_{22}\}$. It follows that   
  
\begin{align*}
P(R)=P\left( (F_{11}\cap F_{21})\cup (F_{11}\cap F_{22})\cup (F_{12}\cap F_{21})\cup (F_{12}\cap F_{22})\right).
\end{align*}
It can be seen, for example, that 

\begin{align*}
B_{(1,1)}\cap B_{(2,2)}&=(F_{11}\cap F_{21})\cap (F_{12}\cap F_{22})\\
&=F_{11}\cap F_{12} \cap F_{21}\cap F_{22}\\
&= (F_{11}\cap F_{21})\cap (F_{11}\cap F_{22})\cap  (F_{12}\cap F_{21})\\
&=B_{(1,1)}\cap B_{(1,2)}\cap B_{(2,1)}.
\end{align*}
Therefore, it seems natural to determine which combinations  lead  to the same intersection set and then simplify the formula. 
The result we are looking for, which will be proved later in the paper, is the following:

\begin{prop}\label{prop1}\ 
Let $\mathcal F=\{F_1,\ldots, F_n\}$ and let $\mathcal F_i=\{F_{i1},\ldots, F_{it_i}\}$, $i\in\{1,\ldots,n\},$ be sets of events such that $F_i=\bigcup_{j=1}^{t_i}F_{ij}$. Let $R=\bigcap_{i=1}^n F_i$,  $m=\sum_{i=1}^n t_i$ and, given $k\in\{n,n+1,\ldots,m\}$, let 

\begin{align*}
 	C_k=&\{\mathbf E=\{E_1,\ldots,E_k\}:\ E_u=F_{i(u)	j(u)}\text{ for some } i(u)\in\{1,\ldots,n\},\ j(u)\in\{1,\ldots,t_{i(u)}\},\ u\in\{1,\ldots,k\},\\ &\ \{i(1),i(2),\ldots,i(k)\}=\{1,\ldots,n\}\text{ and } E_u\ne E_v \text{ for } u\ne v \in\{1,\dots,k\}\}.
\end{align*}
Set $C_k$ is the family of sets $\mathbf E=\{E_1,\ldots, E_k\}$ with cardinality $k$ of events of implementations not failing where every function $i,\ i\in\{1,\dots,n\},$ is implemented at least once. 

It holds that 

\begin{align}\label{eq:prop1}
P(R)=\sum_{k=n}^m\left(\left(-1\right)^{k-n}\sum_{\mathbf E\in C_k}P\left(\bigcap_{j=1}^kE_j\right)\right).
\end{align}
\end{prop}

Table~\ref{t:newcal} shows the total number of summands that we obtain when we use the result stated in Proposition~\ref{prop1} for different numbers of functions and implementations. We assume that every function has the same number of implementations. In this case, the number of summands can be calculated by summing over $k\in \{n,\dots,m\}$ the cardinalities of $|C_k|$. 
It holds that

\begin{equation}
\sum_{k=n}^m|C_k|=\prod_{i=1}^n\left(2^{|\mathcal F_i|}-1\right),
\end{equation}
which is proven in Lemma~\ref{lemma:timecompl}. Therefore, the expression given by (\ref{eq:prop1}) has  a exponential computational complexity with a linear exponent. 
  \begin{table}[H]
\centering
\begin{tabular}{c|cccccccccc}
$|\mathcal F|$ &2&2&2&3&3&3&4&4&5&5\\
\hline
  $|\mathcal F_i|$ &2&3&4&2&3&4&2&3&2&3\\
\hline
 Summands&9&49	  &225	  &27&343&3375&81&2401&243&16807
\end{tabular}
\caption{Number of summands in formula (\ref{eq:prop1}).}
\label{t:newcal}
 \end{table} 

When comparing Tables~\ref{t:oldcal} and \ref{t:newcal}, we can see an enormous reduction in the number of terms involved to calculate the same value of $P(R)$.

\section{Proof of Proposition~\ref{prop1}}\label{sec:math}
In this section,  we prove Proposition~\ref{prop1} after stating and proving an auxiliary result (Lemma~\ref{prop2}). As we mentioned earlier (see (\ref{def:repterms})), there are different subsets of $W$ in formula~(\ref{eq:ie}) that lead to the same intersection set and, therefore, the same probability.  The first result of the following lemma enables us to count how many different subsets of $W$ with the same cardinality $t\in\{1,\dots,|W|\}$  lead to the same intersection set. The second result of   Lemma~\ref{prop2} gives us the coefficient of an intersection set $\mathbf E$ in formula~(\ref{eq:prop1}).

\begin{lemma}\label{prop2}\ \\
Let $t,  n\in\mathbb N_+$ and  let $A_1,\ldots,A_n$ be non-empty sets with $A_i\cap A_j=\emptyset$ $\forall i\ne j$, $A=\bigcup_{i=1}^nA_i$, $k=|A|$ and  ${D= A_1\times \hdots \times A_n}$. Furthermore, let us  define$ \ {s(e)=\{e_1,\ldots,e_n\}}$ for  ${e=(e_1,\ldots,e_n)}\in D$ and, given ${I\subseteq A}$ and $\mathcal A=(A_1,\ldots,A_n)$, let ${p(I,\mathcal A)=\prod_{i=1}^n|A_i\cap I|= | (A_1 \cap I) \times \hdots \times (A_n \cap I)|
}$. \\
Let $c(\mathcal A,t)$ be the total number of non-empty subsets $\{e^1,\ldots,e^t\}\subseteq D$ of cardinality $t$  such that  \\${\bigcup_{\ell=1}^ts(e^\ell)=A}$.  Then 
\begin{center}
\begin{enumerate}
\item  $c(\mathcal A,t)=\sum_{i=0}^{k-n}(-1)^i\sum_{I\subseteq A,\atop |I|=k-i}\binom{p(I,\mathcal A)}{t}, \qquad\refstepcounter{equation}(\theequation)\label{eq:prop2.1}$\\
\item $\sum_{t=1}^{|D|}(-1)^{t-1}c(\mathcal A,t)=(-1)^{k-n}.\quad\qquad\qquad\refstepcounter{equation}(\theequation)\label{eq:prop2.2}$
\end{enumerate}
\end{center}
\end{lemma}
\reqnomode
\begin{proof}\ \\
1.) First we prove (\ref{eq:prop2.1}) by induction over $k\ge n$. We use $\binom i j=0$ if $j>i$ several times throughout the proof. The first time it is needed is to see that  for all $t$ with $t>|D|$, we have that $c(\mathcal A, t)=0$ because there exist no subsets of $D$  with cardinality strictly greater than $|D|$. Furthermore, if $t>|D|$ and because $p(I,\mathcal A)\le |D|$, the right side of (\ref{eq:prop2.1}) is 0. Therefore, (\ref{eq:prop2.1}) holds for  $t>|D|$. In the following, we assume that $t\le|D|$. 

If  $k=n$, then $|A_1|=\ldots=|A_n|=1$ and $|D|=1$ and we can assume $t=1$. Thus

\begin{align*}
\sum_{i=0}^{k-n}(-1)^i\sum_{I\subseteq A,\atop |I|=k-i}\binom{p(I,\mathcal A)}{t}=(-1)^0\binom{1}{1}=1.
\end{align*}
Moreover, $|D|=1$ means that there exists exactly one vector $e\in D$ and $s(e)=A$. Therefore, $c(\mathcal A,1)=1$ and the formula is correct for $k=n$.

Let us  assume that (\ref{eq:prop2.1}) holds for $n,k,t\in \mathbb N_+$  with $k\ge n$. We will show that it also holds for $k+1$. For this, let $A_1,\ldots,A_n$ be non-empty sets with   $A_i\cap A_j=\emptyset$ for all $i\ne j$, $A=\bigcup_{i=1}^nA_i$ with $|A|=k+1$ and $\mathcal A=(A_1,\ldots,A_n)$. We can write the number of non-empty subsets $\{e^1,\ldots,e^t\}\subseteq D$ of  cardinality $t$ from $D$ with $\bigcup_{l=1}^ts(e^l)=A$ as the number of subsets $\{e^1,\ldots,e^t\}\subseteq D$ of cardinality $t$  minus the number of  subsets $\{e^1,\ldots,e^t\}\subseteq D$ of cardinality $t$ with $\bigcup_{i=1}^ts(e^i)=A\backslash J$ for all non-empty sets $J\subseteq A,\ |J|\le k+1-n$ (because $|A\backslash J|\ge n $) and  $p(A\backslash J,\mathcal A)=\prod_{i=1}^n\left|\left(A_i\cap \left(A\backslash J\right)\right)\right|=\prod_{i=1}^n\left|\left(A_i\backslash J\right)\right|\ne 0$. Since for these sets $A\backslash J$ we have that $|A\backslash J|\le k$, we can use the induction hypothesis to obtain that the number of such sets, for each $J$, is $c(\mathcal A\backslash J,t)$, where we define   $\mathcal A\backslash J=(A_1\backslash J,\ldots,A_n\backslash J)$. Therefore, using that  $|D|=p(A,\mathcal A)$ we have  that

\begin{align*}
c(\mathcal A,t)&=\binom{|D|}{t}- \sum_{j=1}^{k+1-n}\left(\sum_{J\subseteq A,\ |J|=j,\atop p(A\backslash J,\mathcal A)\ne 0}c(\mathcal A\backslash J,t)\right)\\
&= \binom{p(A,\mathcal A)}{t}- \sum_{j=1}^{k+1-n}\left(\sum_{J\subseteq A,\ |J|=j,\atop p(A\backslash J,\mathcal A)\ne 0}\left(\sum_{i=0}^{k+1-j-n}(-1)^i\sum_{I\subseteq A\backslash J, \atop |I|=k+1-j-i}\binom{p(I,\mathcal A\backslash J)}{t}\right)\right).\\
\end{align*}

We can drop the condition $p(A\backslash J,\mathcal A)\ne 0$, because if $p(A\backslash J,\mathcal A)= 0$ then there exists $ i\in\{1,\ldots,n\}$ such  that $\ A_i\backslash J=\emptyset$ and therefore  we have for all $ I\subseteq A\backslash J$ that  $p(I,\mathcal A\backslash J)=0$. Thus, by dropping the condition only zeros are added to the sum. 

Furthermore based on the definition of  function $p$, we know that for all $I,J\subseteq A:$ $$p(I,\mathcal A\backslash J)=\prod_{i=1}^n\left|\left(A_i\backslash J\right)\cap I\right|=\prod_{i=1}^n\left|A_i\cap \left(I\backslash J\right)\right|=p(I\backslash J, \mathcal A).$$  Therefore 

\begin{align}\nonumber
c(\mathcal A,t)=& \binom{p(A,\mathcal A)}{t}- \sum_{j=1}^{k+1-n}\left(\sum_{J\subseteq A,\ |J|=j}\left(\sum_{i=0}^{k+1-j-n}(-1)^i\sum_{I\subseteq A\backslash J,\atop |I|=k+1-j-i}\binom{p(I\backslash J,\mathcal A)}{t}\right)\right)\\\nonumber
=& \binom{p(A,\mathcal A)}{t}- \sum_{j=1}^{k+1-n}\left(\sum_{J\subseteq A,\ |J|=j}\left(\sum_{i=0}^{k+1-j-n}(-1)^{i}\sum_{I\subseteq A,\ J\subseteq I,\atop |I|=k+1-i}\binom{p(I\backslash J,\mathcal A)}{t}\right)\right)\\\label{eq:number1}
=& \binom{p(A,\mathcal A)}{t}- \sum_{j=1}^{k+1-n}\left(\sum_{i=0}^{k+1-j-n}(-1)^{i}\left(\sum_{J\subseteq A,\ |J|=j}\sum_{I\subseteq A,\ J\subseteq I,\atop |I|=k+1-i}\binom{p(I\backslash J,\mathcal A)}{t}\right)\right).
\end{align}
Next, we define $A[\ell]=\{I\subseteq A\ :\ |I|=\ell\},\ \ell\in \mathbb N_+,$ and  we have that\\  

$\forall \ell\in \{1,\ldots, |A|-1\}\ \forall L\in A[\ell]\ \forall j\in\{1,\ldots,|A|-\ell\}:$

\begin{align}\nonumber
&\quad\ |\{(I,J)\in A[\ell+j]\times A[j] \ : \ J\subseteq I, \ I\backslash J=L\}|\\\nonumber
&=|\{(I,J)\in A[\ell+j]\times A[j] \ : \ L\cap J=\emptyset, \ I=L\cup J\}|\\\nonumber
&=|\{(L\cup J,J)\in A[\ell+j]\times A[j]\ : \  L\cap J=\emptyset\}|\\\nonumber
&=|\{J\in A[j]\ : \ J\subseteq A\backslash L\}|\\\label{eq:itcan}
&=\binom{|A|-\ell}{j}.
\end{align}
The last step is known from the general result for unordered sampling without replacement in combinatorics. (\ref{eq:itcan}) can now be  used to rewrite (\ref{eq:number1}) by summing over sets $L=I\backslash J$  with the coefficients $\binom{|A|-\ell}{j}$ instead of summing over $J$ and $I$ separately. We can write now that 

\begin{align}\nonumber
 (\ref{eq:number1})=& \binom{p(A,\mathcal A)}{t}- \sum_{j=1}^{k+1-n}\left(\sum_{\ell=0}^{k+1-j-n}(-1)^{\ell}\sum_{L\subseteq A,\atop |L|=k+1-j-\ell}\left(\binom{|A|-(k+1-j-\ell)}{j}\binom{p(L,\mathcal A)}{t}\right)\right)\\\label{eq:number20}
=& \binom{p(A,\mathcal A)}{t}- \sum_{j=1}^{k+1-n}\left(\sum_{\ell=0}^{k+1-j-n}(-1)^{\ell}\sum_{L\subseteq A,\atop |L|=k+1-j-\ell}\left(\binom{j+\ell}{j}\binom{p(L,\mathcal A)}{t}\right)\right).
\end{align}
Let now $\hat \ell \in\{1,\ldots,k+1-n\}$ and  $L\in A[k+1-\hat \ell]$. It holds that 

\begin{equation*}
\forall j\in\{1,\ldots,\hat \ell\}\ \exists \ell\in\{0,\ldots,k+1-n-j\}:\ j+\ell=\hat \ell
\end{equation*} 
 and that $L$ appears exactly once for all possible combinations $j+\ell=\hat \ell$ with the coefficients $(-1)^{\hat \ell -j}\binom{\hat \ell}{j}=(-1)^{\hat \ell +j}\binom{\hat \ell}{j}$  in (\ref{eq:number20}). 
Therefore, we can take the sum over the sets $L\in A[k+1-\hat \ell]$ for $\hat \ell \in\{1,\ldots,k+1-n\}$ and it holds that 

\begin{align}\nonumber
(\ref{eq:number20})=& \binom{p(A,\mathcal A)}{t}- \sum_{\hat \ell=1}^{k+1-n}\left(\sum_{L\subseteq A,\  p(L,\mathcal A)\ne 0,\atop |L|=k+1-\hat \ell}\left(\sum_{j=1}^{\hat \ell}(-1)^{\hat \ell+j}\binom{\hat \ell}{j}\right)\binom{p(L,\mathcal A)}{t}\right)\\\label{eq:number2}
=& \binom{p(A,\mathcal A)}{t}- \sum_{\hat \ell=1}^{k+1-n}\left(\sum_{L\subseteq A,\  p(L,\mathcal A)\ne 0,\atop |L|=k+1-\hat \ell}(-1)^{\hat \ell}\left(\sum_{j=1}^{\hat\ell}(-1)^{j}\binom{\hat \ell}{j}\right)\binom{p(L,\mathcal A)}{t}\right).
\end{align}
Finally, by  using the following known result from combinatorics

\begin{align}\label{eq:bineq0}
\sum_{i=0}^n(-1)^i\binom{n}{i}=\sum_{i=1}^n(-1)^i\binom{n}{i}+1=0,
\end{align}
 we have that

\begin{align*}
 (\ref{eq:number2}) =& \binom{p(A,\mathcal A)}{t}- \sum_{\hat \ell=1}^{k+1-n}\left(\sum_{L\subseteq A,\atop |L|=k+1-\hat\ell}(-1)^{\hat\ell}\left(-1\right)\binom{p(L,\mathcal A)}{t}\right)\\
=& \binom{p(A,\mathcal A)}{t}+ \sum_{\hat\ell=1}^{k+1-n}\left(\sum_{L\subseteq A,\atop |L|=k+1-\hat\ell}(-1)^{\hat\ell}\binom{p(L,\mathcal A)}{t}\right)\\
=& \sum_{\hat \ell=0}^{k+1-n}\left(\sum_{L\subseteq A,\atop |L|=k+1-\hat\ell}(-1)^{\hat\ell}\binom{p(L,\mathcal A)}{t}\right).
\end{align*}
This completes the proof for  (\ref{eq:prop2.1}). \\

2.) Now we will prove the second result, (\ref{eq:prop2.2}), for $n\ge 1$ and $ k\ge 1$ with  $n\le k$ by induction over $m=k-n$. Let $n\ge 1$, $ k\ge 1$ and let $A_1,\ldots,A_n$ be non-empty sets with   $A_i\cap A_j=\emptyset$  for all $ i\ne j$, $A=\bigcup_{i=1}^nA_i$ with $|A|=k$,  $\mathcal A=(A_1,\ldots,A_n)$  and  ${D= A_1\times \hdots \times A_n}$. We can rewrite (\ref{eq:prop2.2}) by using (\ref{eq:prop2.1}) as follows:

\begin{align}\nonumber
\quad \sum_{t=1}^{|D|}(-1)^{t-1}c(\mathcal A,t)&=\sum_{t=1}^{p(A,\mathcal A)}\left((-1)^{t-1}\sum_{i=0}^{k-n}\left(\sum_{I\subseteq A,\atop |I|=k-i}(-1)^i\binom{p(I,\mathcal A)}{t}\right)\right)\\\nonumber
&=\sum_{t=1}^{p(A,\mathcal A)}\left((-1)^{t-1}\sum_{i=0}^{k-n}\left(\sum_{I\subseteq A, \ p(I,\mathcal A)\ne0,\atop |I|=k-i}(-1)^i\binom{p(I,\mathcal A)}{t}\right)\right)\\
\label{eq:usingbineq0}
&=\sum_{i=0}^{k-n}(-1)^i\left(\sum_{I\subseteq A, \ p(I,\mathcal A)\ne0,\atop |I|=k-i}\sum_{t=1}^{p(I,\mathcal A)}(-1)^{t-1}\binom{p(I,\mathcal A)}{t}\right).
\end{align}
Next we use  (\ref{eq:bineq0}) again and it holds that  

\begin{align}\nonumber
 (\ref{eq:usingbineq0})&=\sum_{i=0}^{k-n}(-1)^i\left(\sum_{I\subseteq A, \ p(I,\mathcal A)\ne0,\atop |I|=k-i}(-1)\left(\sum_{t=0}^{p(I,\mathcal A)}(-1)^{t}\binom{p(I,\mathcal A)}{t}-1\right)\right)\\\nonumber
&=\sum_{i=0}^{k-n}(-1)^i\left(\sum_{I\subseteq A, \ p(I,\mathcal A)\ne0,\atop |I|=k-i}(-1)^2\right).\\\nonumber
&=\sum_{i=0}^{k-n}(-1)^i\left(\sum_{I\subseteq A, \ p(I,\mathcal A)\ne0,\atop |I|=k-i}1\right).\\\nonumber
\shortintertext{If we now modify the external sum on the previous expression to start with $i=n$, it follows that}\label{eq:number3}
 \sum_{t=1}^{|D|}(-1)^{t-1}c(\mathcal A,t)&=(-1)^k\sum_{i=n}^{k}(-1)^{i}\left(\sum_{I\subseteq A, \ p(I,\mathcal A)\ne0,\atop |I|=i}1\right).
\end{align}
If $m=0$ and therefore $n=k$, then 

\begin{align*}
\sum_{t=1}^{|D|}(-1)^{t-1}c(\mathcal A,t)&=(-1)^k\sum_{i=n}^{k}(-1)^{i}\left(\sum_{I\subseteq A, \ p(I,\mathcal A)\ne0,\atop |I|=i}1\right)\\
&=(-1)^n\cdot(-1)^n\cdot1=1=(-1)^{k-n}.
\end{align*}
Therefore, equation  (\ref{eq:prop2.2}) holds for $m=0$.

We assume now that it holds for $n\ge 1$ and $ k\ge 1$  with $m=k-n$ and  $m\ge0$. Without loss of generality, we show that it also holds for $m+1$ with $m+1=(k+1)-n$ by fixing $n$. 
For this, let $A_1,\ldots, A_n$ be non-empty sets with $A_i\cap A_j=\emptyset$ for all $ i\ne j$, and $A=\bigcup_{i=1}^nA_i$ with $|A|=\sum_{i=1}^n|A_i|=k$. Furthermore,  let $A^*_1,A^*_2,\ldots,A^*_n$ be non-empty sets with $A^*=\bigcup_{i=1}^nA^*_i$ and $|A^*|=k+1$. Without loss of generality, let $A_2=A^*_2,\ldots,\ A_n=A^*_n$  and $A_1=A^*_1\backslash\{\delta\}$ with $\delta\notin\bigcup_{i=2}^nA_i$. Also, let $\hat A_1,\ldots, \hat A_{n-1}$ be non-empty sets with $\hat A_1=A_2,\ldots, \hat A_{n-1}=A_n$ and $\hat A=\bigcup_{i=1}^{n-1}\hat  A_i$. 

By using (\ref{eq:number3}), it holds  that 

\begin{align*}
&\quad \sum_{t=1}^{p(A^*,\mathcal A^*)}(-1)^{t-1}c(\mathcal A^*,t)\\
&=(-1)^{k+1}\sum_{i=n}^{k+1}(-1)^{i}\left(\sum_{I\subseteq A^*, \ p(I,\mathcal A^*)\ne0,\atop |I|=i}1\right)\\
&=(-1)\left(\left(-1\right)^k\sum_{i=n}^{k+1}(-1)^i\left(\sum_{I\subseteq A^*, \ p(I,\mathcal A^*)\ne0,\atop |I|=i}1\right)\right).\\
\shortintertext{The sum over the subsets $I$ can be split by considering whether  the subsets contain $\delta$ or not:}
&=(-1)\left(\left(\left(-1\right)^k\sum_{i=n}^{k+1}(-1)^i\left(\sum_{I\subseteq A^*, \ p(I,\mathcal A^*)\ne0,\atop |I|=i,\ \delta\notin I}1\right)\right)+\left(\left(-1\right)^k\sum_{i=n}^{k+1}(-1)^i\left(\sum_{I\subseteq A^*, \ p(I,A^*)\ne0,\atop |I|=i,\ \delta \in I}1\right)\right)\right).\\
&=(-1)\left(\left(\left(-1\right)^k\sum_{i=n}^{k+1}(-1)^i\left(\sum_{I\subseteq A^*\backslash \{\delta\}, \ p(I,\mathcal A^*\backslash \{\delta\})\ne0,\atop |I|=i}1\right)\right)+\left(\left(-1\right)^k\sum_{i=n}^{k+1}(-1)^i\left(\sum_{I\subseteq A^*, \ p(I,A^*)\ne0,\atop |I|=i,\ \delta \in I}1\right)\right)\right).\\
\shortintertext{ Using that $A^*\backslash\{\delta\}=A$ and that no subset of $A$ can be of cardinality $k+1$, we can rewrite the first sum: }
&=(-1)\left(\left(\left(-1\right)^k\sum_{i=n}^{k}(-1)^i\left(\sum_{I\subseteq A, \ p(I,\mathcal A)\ne0,\atop |I|=i}1\right)\right)+\left(\left(-1\right)^k\sum_{i=n}^{k+1}(-1)^i\left(\sum_{I\subseteq A^*, \ p(I,\mathcal A^*)\ne0,\atop |I|=i,\ \delta\in I}1\right)\right)\right)\\
&=(-1)\left(\sum_{t=1}^{p(A,\mathcal A)}(-1)^{t-1}c(\mathcal A,t)+\left(\left(-1\right)^k\sum_{i=n}^{k+1}(-1)^i\left(\sum_{I\subseteq A^*, \ p(I,\mathcal A^*)\ne0,\atop |I|=i,\ \delta\in I}1\right)\right)\right),\\\shortintertext{ where we use (\ref{eq:number3}) to obtain this last inequality. Finally, by using  the induction hypothesis for $m=k-n$: }
&=(-1)\left((-1)^{k-n}+\left(\left(-1\right)^k\sum_{i=n}^{k+1}(-1)^i\left(\sum_{I\subseteq A^*, \ p(I,\mathcal A^*)\ne0,\atop |I|=i,\ \delta\in I}1\right)\right)\right)\\
&=(-1)^{k+1-n}+(-1)\left(\left(-1\right)^k\sum_{i=n}^{k+1}(-1)^i\left(\sum_{I\subseteq A^*, \ p(I,\mathcal A^*)\ne0,\atop |I|=i,\ \delta\in I}1\right)\right).
\end{align*}
To proof that (\ref{eq:prop2.2})  holds for $m+1$, we only have to show that 

\begin{align}\label{eq:prop2.2.m}
\left(-1\right)^k\sum_{i=n}^{k+1}(-1)^i\left(\sum_{I\subseteq A^*, \ p(I,\mathcal A^*)\ne0,\atop |I|=i,\ \delta\in I}1\right)=0.
\end{align}
First, we rewrite the left-hand side of (\ref{eq:prop2.2.m}) as follows:

\begin{align}\nonumber
&\left(-1\right)^k\sum_{i=n}^{k+1}(-1)^i\left(\sum_{I\subseteq A^*, \ p(I,\mathcal A^*)\ne0,\atop |I|=i,\ \delta\in I}1\right)\\\nonumber
=&\left(-1\right)^k\sum_{i=n}^{k+1}(-1)^i\left(\sum_{I\subseteq A^*\backslash\{\delta\}, \ p(I\cup\{\delta\},\mathcal A^*)\ne0,\atop |I|=i-1}1\right)\\\label{eq:prop2.2.m1}
=&\left(-1\right)^k\sum_{i=n-1}^{k}(-1)^{i+1}\left(\sum_{I\subseteq A^*\backslash\{\delta\}, \ p(I\cup\{\delta\},\mathcal A^*)\ne0,\atop |I|=i}1\right).
\end{align}
The following observations are needed to rewrite  (\ref{eq:prop2.2.m1}) further.

\begin{enumerate}
\item First note that  that $A^*\backslash A^*_1=\hat A$, and that for all $I\subseteq \hat A:$

\begin{align}\label{eq:number4}
 p(I\cup\{\delta\},\mathcal A^*)=\prod_{i=1}^n\left|A^*_i\cap (I\cup\{\delta\})\right|=|\{\delta\}|\cdot\prod_{i=2}^n\left|A^*_i\cap I\right|=\prod_{i=1}^{n-1}|\hat A_i\cap I|=p(I,\mathcal{ \hat A}).
\end{align}
\item In addition  $A^*\backslash \{\delta\}=A$ and for all $I\subseteq A$ it holds that  

\begin{align}\nonumber
 p(I\cup\{\delta\},\mathcal A^*)&=\prod_{i=1}^n|A^*_i\cap \left(I\cup \{\delta\}\right)|=(|A_1 \cap I| + 1) \prod_{i=2}^n |A_i^* \cap I|\\
 \shortintertext{ and therefore }\nonumber
p(I\cup\{\delta\},\mathcal A^*)&\ne0 \Leftrightarrow\prod_{i=2}^n|A^*_i\cap I|=\prod_{i=1}^{n-1}|\hat A_i\cap I|=p(I,\mathcal{ \hat A})\ne0.
\end{align}
\item Moreover, it holds that

\begin{align}\label{eq:number6}
\forall I\subseteq A:\ p(I,\mathcal{ \hat A})\ne0\ \land \  |I\cap A_1|\ne 0 \Leftrightarrow p(I,\mathcal{ \hat A})\cdot |I\cap A_1|=p(I,\mathcal A)\ne0.
\end{align}
\end{enumerate}
Now we can rewrite (\ref{eq:prop2.2.m1}) by splitting the sum over subsets $I$ by considering whether or not the subset is disjoint with $A_1$.

\begin{align}\nonumber
&\quad \left(-1\right)^k\sum_{i=n-1}^{k}(-1)^i\left(\sum_{I\subseteq A^*\backslash\{\delta\}, \ p(I\cup\{\delta\},\mathcal A^*)\ne0,\atop |I|=i}1\right)\\\label{eq:number5}
&=\left(-1\right)^k\sum_{i=n-1}^{k}(-1)^i\left(\sum_{I\subseteq A^*\backslash A^*_1, \ p(I\cup \{\delta\},\mathcal A^*)\ne0,\atop |I|=i}1\right)+\left(-1\right)^k\sum_{i=n}^{k}(-1)^i\left(\sum_{I\subseteq A^*\backslash\{\delta\}, \  p(I\cup \delta,\mathcal A^*)\ne0,\atop |I|=i,\ I\cap A_1\ne \emptyset}1\right).\\
\shortintertext{ By using (\ref{eq:number4}), we can rewrite  the sum over $I\subseteq A^*\backslash A_1^*=\hat A$ and using (\ref{eq:number6}) the sum over $I\subseteq A^*\backslash \{\delta\}=A$. Furthermore, we can change the upper limit of the first  sum in (\ref{eq:number5}) to $|\hat A|=k-|A_1|$. Therefore, we can write that (\ref{eq:number5}) is}\nonumber
&=\left(-1\right)^k\sum_{i=n-1}^{k-|A_1|}(-1)^i\left(\sum_{I\subseteq \hat A, \ p(I,\hat A)\ne0,\atop |I|=i}1\right)+\left(-1\right)^k\sum_{i=n}^{k}(-1)^i\left(\sum_{I\subseteq A, \ p(I,A)\ne0,\atop |I|=i}1\right)\\\label{eq:number7}
&=\left(-1\right)^{|A_1|}\left(-1\right)^{k-|A_1|}\sum_{i=n-1}^{k-|A_1|}(-1)^i\left(\sum_{I\subseteq \hat A, \ p(I,\hat A)\ne0,\atop |I|=i}1\right)+\left(-1\right)^k\sum_{i=n}^{k}(-1)^i\left(\sum_{I\subseteq A, \ p(I,A)\ne0,\atop |I|=i}1\right).\\
\shortintertext{ And  because $(k-|A_1|)-(n-1)<m+1$ and $k-n<m+1$, we can use the induction hypothesis on both sums and we use  (\ref{eq:number3}) to obtain that  }\nonumber\\\nonumber
 (\ref{eq:number7})&=\left(\left(-1\right)^{|A_1|}\left(-1\right)^{k-|A_1|-(n-1)}+\left(-1\right)^{k-n}\right)\\\nonumber
&=\left((-1)^{k-(n-1)}+(-1)^{k-n}\right)=0.
\end{align}
Hence, we have proven that 

\begin{align*}
\sum_{t=1}^{|D|}(-1)^{t-1}c(\mathcal A,t)=(-1)^{k-n}
\end{align*}
 holds for any $n\ge1$ and $k\ge1$ with $k-n\ge0$. 
\end{proof}

We can now  prove our main result Proposition~\ref{prop1}.

\begin{proof}[Proof of  Proposition~\ref{prop1}]\ \\ \ \\
Since  $R=\bigcap_{F\in\mathcal F}F$, it follows that

\begin{align}\label{eq:prof1_1}
P(R)=P\left(\bigcap_{i=1}^n F_i \right)=P\left(\bigcap_{i=1}^n\left(\bigcup_{j=1}^{t_i} F_{ij}\right)\right).
\end{align}
Let 

\begin{align*}
W&=\{1,\ldots,t_1\}\times\hdots\times\{1,\ldots,t_n\},\\
 B_w&=\bigcap_{i=1}^n F_{iw_i}\  \text{ and }\\
 \mathfrak B_w&=\{ F_{1w_1},\ldots, F_{nw_n}\}\text{ for } w=(w_1,\ldots,w_n)\in W.
 \end{align*}
 We can rewrite (\ref{eq:prof1_1}) as 

\begin{align*}
P(R)&= P\left(\bigcup_{w\in W}\left(\bigcap_{i=1}^n F_{iw_i}\right)\right)= P\left(\bigcup_{w\in W}B_w\right).
\end{align*}
Using the probability principle of inclusion-exclusion (\ref{eq:prob_ie}), it holds that 

  \begin{align}\label{eq:prof1_2}
 P(R)=\sum_{t=1}^{|W|}\left(\left(-1\right)^{t+1}\sum_{I\subseteq W,\atop |I|=t}P\left(\bigcap_{j\in I} B_j\right)\right).
 \end{align}
 Based on the definition of $B_w,\ w\in W$, we know that 
 
 \begin{align}\label{eq:number8}
 \forall I\subseteq W\ \exists k\in\{n,n+1,\ldots,m\}\ \text{ and } \exists \mathbf E=\{E_1,\dots,E_k\}\in C_k: \bigcap_{j\in I} B_j= \bigcap_{i=1}^k E_i.
 \end{align}
We define for all  $k\in \{n,n+1,\ldots,m\}$ and $\mathbf E\in C_k\ :$

\begin{align*}
D_{\mathbf E}&=\{B_w\ :\ \mathfrak B_w\cap \mathbf  E=\mathfrak B_w \text{ and } w\in W\}.
\end{align*} 
Furthermore  for all $k\in \{n,\ldots,m\},\l\mathbf E\in C_k$ and $ \ell\in\{1,\ldots,|D_{\mathbf E}|\}$ let us define

\begin{align*}
T\left(\mathbf E,\ell\right)&=\left\{I\subseteq W\ : \ |I|=\ell \text{ and } \bigcup_{i\in I}\mathfrak B_i= \mathbf E\right\} \text{ and }\\
t\left(\mathbf E,\ell\right)&= \left|T(\mathbf E,\ell)\right|.
 \end{align*}

If we use (\ref{eq:number8}),  we can rewrite (\ref{eq:prof1_2}) to a sum over $\mathbf E\in C_k,\ k=n,\dots,m,$ where the coefficients  are $\sum_{i=1}^{|W|}\left(-1\right)^{i+1}t(\mathbf E,i)$. Furthermore, we can rewrite the coefficient to $\sum_{i=1}^{|D_{\mathbf E|}}\left(-1\right)^{i+1}t(\mathbf E,i)$, because $t(\mathbf E,\ell)$ is zero for $\mathbf E \in C_k, k=n,\dots,m$ and $\ell\ge1$ if $|D_{\mathbf E}|<\ell$. 
Hence, we have that  (\ref{eq:prof1_2}) is 

\begin{align}\label{eq:number9}
  =\sum_{k=n}^m\left(\sum_{E\in C_k}\left(\sum_{i=1}^{|D_E|}\left(-1\right)^{i+1}t(E,i)\right)P\left(\bigcap_{j=1}^kE_j\right)\right)
\end{align}
In addition, let $\mathcal E=\left(\mathcal F_1\cap \mathbf E,\ldots, \mathcal F_n\cap \mathbf E\right)$. Based on the first result (\ref{eq:prop2.1}) from Lemma~\ref{prop2}, we know  that $ t(\mathbf E,\ell)=c\left(\mathcal E,\ell\right)$ for all $k\in \{n,n+1,\ldots,m\},\ \mathbf E\in C_k$ and $ \ell\in\{1,\ldots,|D_{\mathbf E}|\}$. This gives us that (\ref{eq:number8}) is

  \begin{align*}
  &=\sum_{k=n}^m\left(\sum_{E\in C_k}\left(\sum_{i=1}^{|D_E|}\left(-1\right)^{i+1}c(\mathcal E,i)\right)P\left(\bigcap_{j=1}^kE_j\right)\right)\\
    \shortintertext{ and using (\ref{eq:prop2.2}) from Lemma~\ref{prop2}, it holds that  }
  &=\sum_{k=n}^m\left(\sum_{E\in C_k}(-1)^{k-n}P\left(\bigcap_{j=1}^k E_j\right)\right)\\
   &=\sum_{k=n}^m\left((-1)^{k-n}\sum_{E\in C_k}P\left(\bigcap_{j=1}^k E_j\right)\right).
 \end{align*}
 This completes the proof and we have shown  that (\ref{eq:prop1}) holds. 
\end{proof}

\section{Time complexity of  main result}\label{sec:timecompl}
In this section, we will proof the time complexity of the new formula in Proposition~\ref{prop1} through Lemma~\ref{lemma:timecompl}. 
\begin{lemma}\label{lemma:timecompl}
Let $n\ge1$, $\mathcal F_i=\{F_{i1},\dots,F_{it_{i}}\},\ i\in\{1,\dots,n\},$ be sets of cardinality $t_i,\ t_i\ge1$ and $\mathbf F=(\mathcal F_i,\dots,\mathcal F_n)$. Define ${m\coloneqq\sum_{i=1}^nt_i}$ and,  given $k\in\{n,n+1,\ldots,m\}$,

\begin{align*}
 	C_k^{\mathbf F}\coloneqq&\{\mathbf E=\{E_1,\ldots,E_k\}:\ E_u=F_{i(u)	j(u)}\text{ for some } i(u)\in\{1,\ldots,n\},\ j(u)\in\{1,\ldots,t_{i(u)}\},\ u\in\{1,\ldots,k\},\\ &\ \{i(1),i(2),\ldots,i(k)\}=\{1,\ldots,n\}\text{ and } E_u\ne E_v \text{ for } u\ne v \in\{1,\dots,k\}\}.
\end{align*} 
Lasltly, define $\mathfrak c(\mathbf F)=\sum_{k=n}^m|C_k^{\mathbf  F}|$. It holds that 

\begin{equation}\label{eq:timecomplex1}
\mathfrak c(\mathbf F)=\sum_{k=n}^m|C_k^{\mathbf  F}|=\prod_{i=1}^n\left(2^{|\mathcal F_i|}-1\right).
\end{equation}
\end{lemma}

\begin{proof}
We proof (\ref{eq:timecomplex1}) by induction. Let $n\ge1$ and $t_i\ge1,\ i\in\{1,\dots,n\}$. We use induction on $\ell=m-n$. Because all $t_i,\ i\in\{1,\dots,n\},$ are greater than 1, $m\ge n$ and can first assume $\ell=0$ with  $m=n$.  With ${m=n}$, we have that $t_1=\hdots=t_n=1$. Therefore $C_n=\{\{F_{11},\dots, F_{n1}\}\}$ and ${\mathfrak c(\mathbf F)=\sum_{k=n}^m|C_k|= |C_n|=1}$. Also $\prod_{i=1}^n\left(2^{|\mathcal F_i|}-1\right)=\prod_{i=1}^n1=1$ and  (\ref{eq:timecomplex1}) holds for $\ell=0$.

We assume now that (\ref{eq:timecomplex1}) holds for $n\ge1$ and $t_i\ge1,\ i\in\{1,\dots,n\}$, with $\ell=m-n$ and $\ell\ge0$. Without loss of generality, we show it also holds for $\ell+1$ with $\ell+1=(m+1)-n$ by fixing $n$. Let $n\ge1$, $\mathcal F_i=\{F_{i1},\dots,F_{it_{i}}\},\ i\in\{1,\dots,n\},$ be sets of cardinality $t_i,\ t_i\ge1$ and $m=\sum_{i=1}^nt_i$. We then know that $\mathfrak c(\mathbf F)=\prod_{i=1}^n\left(2^{|\mathcal F_i|}-1\right)$. Without loss of generality, let $\mathcal{ F}^*_i=\mathcal F_i,\ i\in\{1,\dots,n-1\},$, $\mathcal F_n^*=\mathcal F_n \cup \{\delta\},\ \delta\notin \mathcal F_n$,  and $\mathbf F^*=(\mathcal{ F}^*_1,\dots,\mathcal{ F}^*_n)$. Furthermore, let $\hat{\mathcal F_i}=\mathcal F_i,\ i\in\{1,\dots,n-1\}$ and $\hat{\mathbf F}=(\hat{\mathcal F_1},\dots,\hat{\mathcal F_{n-1}})$. For $k\in\{n,\dots,m+1\}$, we can split the set $C_k^{\mathbf F^*}$ by considering whether or not $\delta$ is contained in  $\mathbf E$ we have that

\begin{align*}
C_k^{\mathbf F^*}&=\{\mathbf E\in C_k^{\mathbf F^*}: \delta \in \mathbf E\}\cup\{\mathbf E\in C_k^{\mathbf F^*}: \delta \notin \mathbf E\} \text{ and } \\
|C_k^{\mathbf F^*}|&=|\{\mathbf E\in C_k^{\mathbf F^*}: \delta \in \mathbf E\}|+|\{\mathbf E\in C_k^{\mathbf F^*}: \delta \notin \mathbf E\}|.  
\end{align*}
For $k\in\{n,\dots,m\}$, $\{\mathbf E\in C_k^{\mathbf F^*}: \delta \notin \mathbf E\}=C_k^{\mathbf F}$ and $\{\mathbf E\in C_{m+1}^{\mathbf F^*}: \delta \notin \mathbf E\}=\emptyset$. By the induction hypothesis, 

\begin{align}\label{eq:1set}
\sum_{k=n}^{m+1}|\{\mathbf E\in C_k^{\mathbf F^*}: \delta \notin \mathbf E\}|=\sum_{k=n}^m|C_k^{\mathbf F}|=\prod_{i=1}^n\left(2^{|\mathcal F_i|}-1\right)=\mathfrak c(\mathbf F).
\end{align}
The set $\{\mathbf E\in C_k^{\mathbf F^*}: \delta \in \mathbf E\}$, $k\in\{n,\dots,m+1\}$, can be further split into two sets by considering if for $E \in  C_k^{\mathbf F^*}$ $\mathcal F_n\cap \mathbf E=\emptyset$ or $\mathcal F_n\cap \mathbf E\ne\emptyset$ and have 

\begin{align*}
\{\mathbf E\in C_k^{\mathbf F^*}: \delta \in \mathbf E\}&=\{\mathbf E\in C_k^{\mathbf F^*}: \delta \in \mathbf E \text{ and } \mathcal F_n\cap \mathbf E=\emptyset\}\cup\{\mathbf E\in C_k^{\mathbf F^*}: \delta \in \mathbf E \text{ and } \mathcal F_n\cap \mathbf E\ne\emptyset \} \text{ and } \\
|\{\mathbf E\in C_k^{\mathbf F^*}: \delta \in \mathbf E\}|&=|\{\mathbf E\in C_k^{\mathbf F^*}: \delta \in \mathbf E \text{ and } \mathcal F_n\cap \mathbf E=\emptyset\}|+|\{\mathbf E\in C_k^{\mathbf F^*}: \delta \in \mathbf E \text{ and } \mathcal F_n\cap \mathbf E\ne\emptyset \}|. 
\end{align*}
Let $\hat m=\sum_{i=1}^{n-1}t_i$. The set $\{\mathbf E\in C_k^{\mathbf F^*}: \delta \in \mathbf E \text{ and } \mathcal F_n\cap \mathbf E=\emptyset\}$, $k\in\{n,\dots,\hat m+1\}$, can also be rewritten as

\begin{align*}
\{\mathbf E\in C_k^{\mathbf F^*}: \delta \in \mathbf E \text{ and } \mathcal F_n\cap \mathbf E=\emptyset\}=\{\mathbf E\cup\{\delta\}: \mathbf E\in C_{k-1}^{\hat{\mathbf F}}\}. 
\end{align*} Therefore it holds that 

\begin{align*} |\{\mathbf E\in C_k^{\mathbf F^*}: \delta \in \mathbf E \text{ and } \mathcal F_n\cap \mathbf E=\emptyset\}|=|\{\mathbf E\cup\{\delta\}: \mathbf E\in C_{k-1}^{\hat{\mathbf F}}\}|=|C_{k-1}^{\hat{\mathbf F}}|.
\end{align*}
Next, since $|\left(\bigcup_{i=1}^n \mathcal F^*_i\right)\cap\mathcal F_n|=\hat m+1$, we have that $\{\mathbf E\in C_k^{\mathbf F^*}: \delta \in \mathbf E \text{ and } \mathcal F_n\cap \mathbf E=\emptyset\}=\emptyset$ for  $k\in\{\hat m+2,\dots, m+1\}$. Because $\hat m+1-(n-1)< (m+1)-n=\ell+1$, we can use the induction hypothesis and it holds that 

\begin{align}\label{eq:2set}
\sum_{k=n}^{m+1}|\{\mathbf E\in C_k^{\mathbf F^*}: \delta \in \mathbf E \text{ and } \mathcal F_n\cap \mathbf E=\emptyset\}|=\sum_{k=n}^{\hat m+1}|C_{k-1}^{\hat{\mathbf F}}|=\sum_{k=n-1}^{\hat m}|C_{k}^{\hat{\mathbf F}}|=\prod_{i=1}^{n-1}\left(2^{|\hat{\mathcal F}_i|}-1\right)=\mathfrak c(\hat{\mathbf F}).
\end{align}
Now we consider the sets $\{\mathbf E\in C_k^{\mathbf F^*}: \delta \in \mathbf E \text{ and } \mathcal F_n\cap \mathbf E\ne\emptyset\}$, $k\in\{n,\dots,m\}$. For $k=n$, let $\mathbf E\in C_k^{\mathbf F^*}$ and we have that $|\mathbf E\cap \mathcal F_i^*|=1$ for $i\in\{1,\dots,n\}$. Therefore, $\{\mathbf E\in C_n^{\mathbf F^*}: \delta \in \mathbf E \text{ and } \mathcal F_n\cap \mathbf E\ne\emptyset\}=\emptyset$. 

For $k\in\{n+1,\dots,m+1\}$ it holds that 

\begin{align*}
\{\mathbf E\in C_k^{\mathbf F^*}: \delta \in \mathbf E \text{ and } \mathcal F_n\cap \mathbf E\ne\emptyset\}= \{\mathbf E\cup\{\delta\}:\mathbf E\in C_{k-1}^{\mathbf F}\}
\end{align*}
and it follows that

\begin{align}\label{eq:3set}
 |\{\mathbf E\in C_k^{\mathbf F^*}: \delta \in \mathbf E \text{ and } \mathcal F_n\cap \mathbf E\ne\emptyset\}|=|\{\mathbf E\cup\{\delta\}: \mathbf E\in C_{k-1}^{\mathbf F}\}|=|C_{k-1}^{\mathbf F}|.
\end{align}
By using the induction hypothesis again, 

\begin{align*}
\sum_{k=n}^{m+1}|\{\mathbf E\in C_k^{\mathbf F^*}: \delta \in \mathbf E \text{ and } \mathcal F_n\cap \mathbf E\ne\emptyset\}|=\sum_{k=n+1}^{m+1}|C_{k-1}^{\mathbf F}|=\sum_{k=n}^{m}|C_{k}^{\mathbf F}|=\prod_{i=1}^n\left(2^{|\mathcal F_i|}-1\right)=\mathfrak c(\mathbf F).
\end{align*}
We can now write 

\begin{align*}
\mathfrak c(\mathbf F^*)&=\sum_{k=n}^m|C_k^{\mathbf  F^*}|\\
&=\sum_{k=n}^m|\{\mathbf E\in C_k^{\mathbf F^*}: \delta \notin \mathbf E\}|+\sum_{k=n}^m|\{\mathbf E\in C_k^{\mathbf F^*}: \delta \in \mathbf E \text{ and } \mathcal F_n\cap \mathbf E=\emptyset\}|\\&\ \ +\sum_{k=n}^m|\{\mathbf E\in C_k^{\mathbf F^*}: \delta \in \mathbf E \text{ and } \mathcal F_n\cap \mathbf E\ne\emptyset\}| 
\shortintertext{ and using (\ref{eq:1set}),  (\ref{eq:2set}) and  (\ref{eq:3set}), we obtain that it is }
&= 2\mathfrak c(\mathbf F)+\mathfrak c(\hat{ \mathbf F})\\&= 2\prod_{i=1}^n\left(2^{|\mathcal F_i|}-1\right)+\prod_{i=1}^{n-1}\left(2^{|\mathcal F_i|}-1\right)\\
&= \left(2\left(2^{|\mathcal F_n|}-1\right)+1\right)\prod_{i=1}^{n-1}\left(2^{|\mathcal F_i|}-1\right)\\&= \prod_{i=1}^{n-1}\left(2^{|\mathcal F_i|}-1\right)*\left(2^{|\mathcal F_n|+1}-1\right)\\
&= \prod_{i=1}^{n}\left(2^{|\mathcal F^*_i|}-1\right).
\end{align*}
This shows that (\ref{eq:timecomplex1}) holds for $\ell+1$ and it completes the proof. 
\end{proof}
\section{Example}\label{sec:example}
In aircraft architecture there are many different types of safety-critical systems such as networked aircraft systems, fire warning systems or stall recovery systems. In this example we consider  the door management system (DMS) of an aircraft which falls into the category of networked aircraft systems.  DMS is a safety-critical system which checks the status of doors, regulates the locks and relays information to on-board computers and pressurization regulators. See Figure~\ref{fig:dms1} for an example of a DMS at one door. 
\begin{figure}[H]
\centering
\includegraphics[width=.35\textwidth]{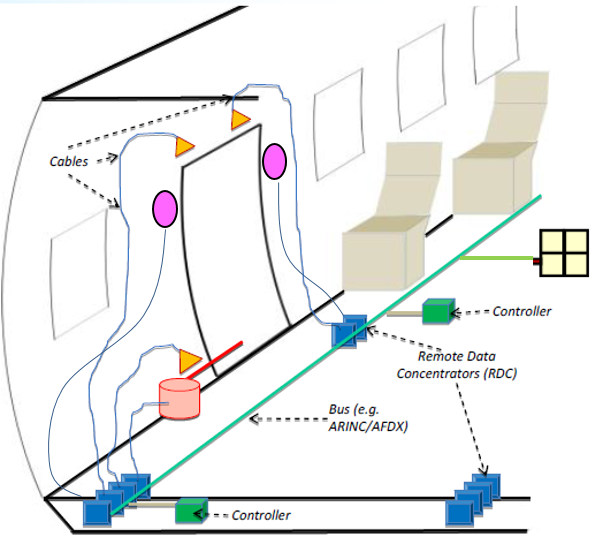}
\caption{Door management system for one door.}
\label{fig:dms1}
\end{figure}
 Figure~\ref{fig:dms1} shows only a part of a DMS  in an aircraft. The DMS is responsible for all doors to the outside in the aircraft as can be seen in Figure~\ref{fig:dms2}. 
 \begin{figure}[H]
\centering
\includegraphics[width=.7\textwidth]{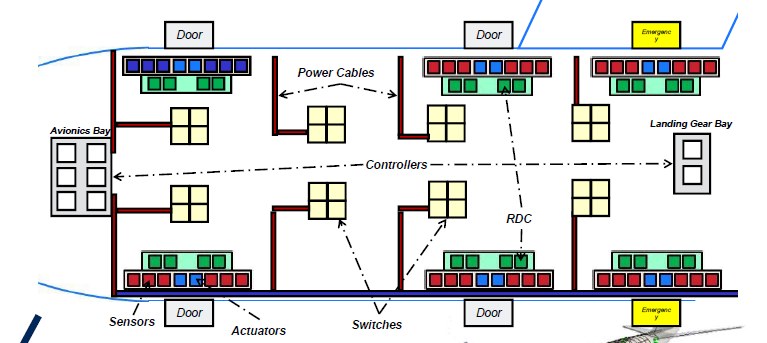}
\caption{Door management system for multiple doors.}
\label{fig:dms2}
\end{figure}
As a safety-critical system which can not be repaired during use,  it must meet some reliability threshold.
As the DMS is responsible for all the doors to the outside in the aircraft and also gives the possibility that different components   (switches, actuators, etc.) can be used by different doors, it is a complex network system for which the calculation of reliability is computationally very expensive as shown in Section~\ref{sec:mr}. However, the computational effort of  the calculation of the reliability of the DMS is reduced considerably by using  Proposition~\ref{prop1}. 

The functionality of the DMS for each door can be seen as a function of the systems. Hence, let the aircraft have $n$ doors and we then have the event set $\mathcal F=\{F_1,\ldots,F_n\}$ where $F_i$ is the event ``The functionality of the DMS for door $i$ in the system does not fail''. Furthermore, let an implementation of the functionality of a door be a set of components and the corresponding connections that can check the status of door, regulate the locks and relay information to on-board computers and pressurization regulators while no subset of these components and connections can be removed without losing functionality. Because the DMS is a safety-critical system that has to be redundant, there are at least two  implementations for every door. Let  $\mathcal F_i=\{F_{i1},\dots,F_{it_i}\}$ and $t_i\ge2$ $,\ i\in\{1,\ldots,n\},$ be event sets where $F_{ij}$ is the event ``The implementation $j$ of door $i$ does not fail''. Lastly, let $R$ be the event ``The DMS system for all doors does not fail''. With these event sets, we can calculate the reliability $P(R)$ of the DMS  with the formula 
\begin{align}\label{eq:prop1.example}
P(R)=\sum_{k=n}^m\left(\left(-1\right)^{k-n}\sum_{E\in C_k}P\left(\bigcap_{j=1}^kE_j\right)\right)
\end{align}
 from Proposition~\ref{prop1}. Furthermore the calculation of $P(\bigcap_{j=1}^kE_j)$ for  $\mathbf E=\{E_1,\dots,E_k\}\in C_k,\ k\in\{n,n+1,\ldots,m\}$ is simple. Let $T_c$ be the event that component $c$ of the DMS system does not fail. Since we are considering a static system, we know the probability of $P(T_c)=a_c$ with $a_c\in(0,1)$. Let $\mathcal T_{F_{ij}}$ be the set of components of implementation $j$ for door $i$. Since we assume that all components have independent failures  for all ${i\in\{1,\ldots,n\}}$ and for all ${ j\in\{1,\ldots,t_i\}}$, then
  \begin{align*}
 P\left(F_{ij}\right)=P\left(\bigcap_{c\in\mathcal  T_{F_{ij}}}T_c\right)=\prod_{c\in \mathcal T_{F_{ij}}}P\left(T_c\right).
\end{align*}
Therefore, for all $ k\in\{n,n+1,\ldots,m\}$ and for all $\mathbf E\in C_k$, it holds that 
\begin{align*}
 P\left(\bigcap_{j=1}^kE_j\right)= P\left(\bigcap_{j=1}^k\left(\bigcap_{c \in \mathcal T_{E_j}}T_c\right)\right)
=  P\left(\bigcap_{c \in \bigcup_{j=1}^k\mathcal T_{E_j}}T_c\right)
=  \prod_{c \in \bigcup_{j=1}^k\mathcal T_{E_j}}P\left(T_c\right).
\end{align*}
If $z$ is  the number of components of the system, we know  that  each summand of (\ref{eq:prop1.example}) is  a product of at most $z$ factors. In  Table~\ref{tab:timecomp}, we show the time in seconds needed to calculate the reliability of DMS systems with different numbers of doors and implementations per door by using Proposition~\ref{prop1} (TNew) and  by using the classic  probability principle of inclusion-exclusion (\ref{eq:ie}) (TOld).  In the header row of Table~\ref{tab:timecomp} we have written the values for pairs ($|\mathcal F|$,$|\mathcal F_i|$) where $|\mathcal F_i|=|\mathcal F_j|$ for all $i,j\in\{1,\dots, |F|\}$. All computations are implemented and run in  R.
\begin{table}[H]
\centering
\begin{tabular}{c|c |c|c|c|c|c|c|c|c|c|c|c|c|c|c|c|c}
($|\mathcal F|$,$|\mathcal F_i|$)&Components&Connections&TNew&TOld\\
\hline
(2,2)&31&47&0.02&0.02\\
(2,2)&31&43&0.02&0.02\\
(2,3)&22&32&0.03&0.12\\
(2,3)&46&63&0.04&0.18\\
(3,2)&37&64&0.05&0.13\\
(3,2)&45&69&0.05&0.14\\
(3,3)&55&94&0.19&$>400$\\
(3,3)&70&102&0.20&$>400$\\
(4,2)&42&71&0.07&29.05\\
(4,2)&55&90&0.08&32.53\\
(4,3)&65&121&1.55&$>400$\\
(4,3)&90&164&1.7081&$>400$\\
(5,2)&50&89&0.17&$>400$\\
(5,3)&53&84&31.07&$>400$
\end{tabular}
\caption{Comparison of computational times for reliability of  different sized DMS.}
\label{tab:timecomp}
 \end{table}
As can be seen, there is a huge reduction on the computational time if we use the expression  (\ref{eq:prop1.example}) that we have introduced in this paper. This difference   increases when the number of doors and implementations does.  Furthermore, it can be observed that the number of components and connections  does not affect the computation time not significantly.
\section{Conclusions}\label{sec:conclusion}

In this paper we have introduced a new expression (Proposition~\ref{prop1}) which reduces considerably  the computational effort needed for the calculation of the probability principle of inclusion-exclusion when applied to  intersections of unions of events. It has been shown that  the formula obtained can be applied to the reliability calculation of   complex network systems and it allows to decrease the computational time significantly. Furthermore, it has also been shown that the  computational complexity is reduced from double exponential to exponential with linear exponent.

This result opens doors  to formulate optimization problems of complex network systems that include the exact reliability of the system without depending solely on heuristics to solve it. 

A potential extension for future research is to  generalize Proposition~\ref{prop1} and Lemma~\ref{prop2}. For example, the implementation for a function $i$ can be also used for function $j$ with $i\ne j$ which results in $\mathcal F_i \cap \mathcal F_j\ne \emptyset$. This can result in a simplification of the probability principle of inclusion-exclusion with summand coefficients that are not $-1$ or $1$ but  still give a decrease on the number of summands compared to Proposition~\ref{prop1}. 
\section*{ Acknowledgements }
The research of Sergio Garc\'{i}a  has been funded by Fundaci\'{o}n S\'{e}neca (project 19320/PI/14). Lukas Sch{\"a}fer is funded by an EPSRC Industrial CASE studentship in partnership with Airbus Group.
\bibliography{literature2}
\appendix
\section{Lower bound of probability principle of inclusion-exclusion }\label{asec:lbce}
When you want to  optimize the reliability of a system, you cannot use lower bounds. To use a lower bound, we need that  the lower bound of the reliability of $T_1$ is lower than the lower bound of the reliability of $T_2$ if the reliability of a system $T_1$ is lower than the reliability of a system $T_2$.  Otherwise, one cannot be sure that the optimal solution you get by using the lower bound is the optimal solution regarding the exact reliability. We did not find a lower bound that is suitable for optimization and fulfills that latter criterion.\\

We show an example of two systems where the reliability of one system is greater than the other, but the lower bound of the reliability of the former is smaller than the latter. This is the reason why it is not suitable to use in optimization. For the example, we use the following lower bound proposed by \cite{daw67}.

Let $A_1,\ldots,A_n$ be the considered events, $S_1=\sum_{k=1}^nP(A_k)$ and $S_2=\sum_{1\le i<j\le n}P(A_i\cap A_j)$. 
 \begin{theor}\cite{daw67}\\
 Given a probability measure space $(\Omega,\mathbb F, P)$, let $A_k\in \mathbb F$, $k=1,\ldots,n$.
 \begin{equation}\label{g:dawlb}
 P\left(\bigcup_{k=1}^nA_k\right)\ge \frac{\theta S_1^2}{2S_2+(2-\theta)S_1}+\frac{(1-\theta)S_1^2}{2S_2+(1-\theta)S_1}
 \end{equation}
 with $\theta =2S_2/S_1-[2S_2/S_1]. $
 \end{theor}
 Because it is not easy to define $\theta$ in an optimization constraint, we consider the minimum of the right-hand side which occurs with $\theta=0$ and we obtain that
  \begin{equation}\label{m:dawlb}
 P\left(\bigcup_{k=1}^nA_k\right)\ge \frac{S_1^2}{2S_2+S_1}.
 \end{equation}
 Calculating the derivatives, it is easy to see that $\theta=0$ gives the minimum of the right-hand side, if the lower bound given by the right-hand side is considered as function $f$ of $\theta$. It is easily  seen that $f$ is a concave function for $0\le\theta\le1$ and $f(0)=f(1)$ . This gives you that $f(0)\le f(\theta)$ for $0\le\theta<1$.
 
 
For the example, let us assume that we have a system with one function and the function is  implemented three times. Let $A$, $B$ and $C$ be the implementations. 
We can calculate for system $T$  the exact reliability $P(T)$  as follows :
\begin{equation*}
P(T)=P(A\cup B\cup C)=P(A)+P(B)+P(C)-P(A\cap B) -  P(A\cap C)-P(B\cap C)+P(A\cap B\cap C).
\end{equation*}
We will also calculate a lower bound based on (\ref{m:dawlb}).
Let $T_1$ and $T_2$ be as seen in  Figures~\ref{fig:T1} and~ \ref{fig:T2}, respectively. Furthermore, let $P(1)=0.5,\ P(2)=0.7,\ P(3)=0.2,\ P(4)=0.6$ and $P(5)=0.3$ be the probabilities of the unit not failing for $T_1$ and take $P(3)=0.2521$ for $T_2$. 
We obtain that $P(T_1)=0.2668$ and that $P(T_2)=0.2668232$. Therefore, the reliability of $T_2$ is greater than the reliability of $T_1$. But if we look at the lower bounds, we have that $P(T_1)_{lb}=0.2260049$ and $P(T_2)_{lb}=0.2257831$. Therefore, we can see that the lower bound of $T_1$ is greater than the lower bound of $T_2$. This shows us that the lower bounds are not monotone increasing with the reliability and therefore should be used for exact optimization of reliability systems. 
Another example is system $T_3$ in Figure~\ref{fig:T3}. Let $1, 2, 4$ and 5 have the same probabilities of not failing as for $T_1$ and  let  $P(3)=0.3$. We obtain that  $P(T_3)=0.261$ for $T_3$ which is smaller than $P(T_1)$ and $P(T_2)$. But the lower bound of $P(T_3)$ is $P(T_3)_{lb}=0.2278481$ which is greater than the lower bound of $P(T_1)$ and $P(T_2)$. 
\begin{figure}[H]
\center
\begin{tikzpicture}[line width=1pt]
\node [rectangle,draw]   (1) at (-5, 4) {1};
\node [rectangle,draw]   (2) at (-3, 4){2};
\node [rectangle,draw]   (3) at (-1, 4){3};
\node [rectangle,draw]   (4) at (-5, 3){4};
\node [rectangle,draw]   (5) at (-3, 3){5};
\node[draw=none] (d1t) at (-10,4) {A};
\node[draw=none] (d1t) at (-10,3) {B };
\node[draw=none] (d1t) at (-10,2) {C};
	\draw[->,color=red]  (1) to (2);
	\draw[->,color=red]([yshift=-6pt,xshift=6pt]2.north)--node{}([yshift=-6pt,xshift=-6pt]3.north);
	\draw[->,color=blue]([yshift=6pt,xshift=6pt]2.south)--node{}([yshift=6pt,xshift=-6pt]3.south);
	\draw[->,color=blue]  (4) to (2);
	\draw[->,color=green]  (4) to (5);
	\draw[red] (-9,4) -- (-7,4);
	\draw[blue] (-9,3) -- (-7,3);
	\draw[green] (-9,2) -- (-7,2);
\end{tikzpicture}
\caption{System $T_1$} 
\label{fig:T1}
\end{figure}
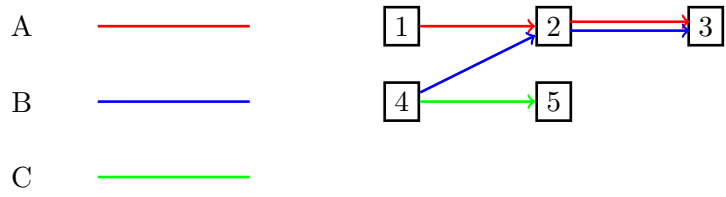
\begin{figure}[H]
\center
\begin{tikzpicture}[line width=1pt]
\node [rectangle,draw]   (1) at (-5, 4) {1};
\node [rectangle,draw]   (2) at (-3, 4){2};
\node [rectangle,draw]   (3) at (-1, 4){3};
\node [rectangle,draw]   (4) at (-5, 3){4};
\node [rectangle,draw]   (5) at (-3, 3){5};
\node[draw=none] (d1t) at (-10,4) {A};
\node[draw=none] (d1t) at (-10,3) {B };
\node[draw=none] (d1t) at (-10,2) {C};
	\draw[->,color=red]([yshift=-6pt,xshift=6pt]1.north)--node{}([yshift=-6pt,xshift=-6pt]2.north);
	\draw[->,color=blue]([yshift=6pt,xshift=6pt]1.south)--node{}([yshift=6pt,xshift=-6pt]2.south);
		\draw[->,color=red]([yshift=-6pt,xshift=6pt]2.north)--node{}([yshift=-6pt,xshift=-6pt]3.north);
	\draw[->,color=blue]([yshift=6pt,xshift=6pt]2.south)--node{}([yshift=6pt,xshift=-6pt]3.south);
	\draw[->,color=green]  (4) to (5);
	\draw[red] (-9,4) -- (-7,4);
	\draw[blue] (-9,3) -- (-7,3);
	\draw[green] (-9,2) -- (-7,2);
\end{tikzpicture}
\caption{System $T_2$} 
\label{fig:T2}
\end{figure}
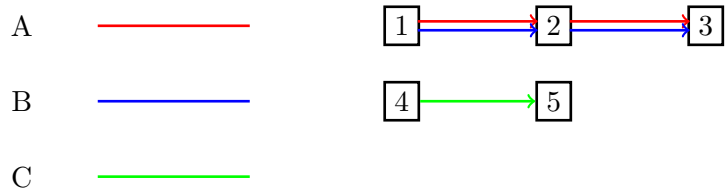
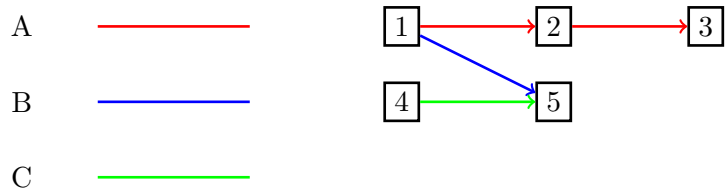
\begin{figure}[H]
\center
\begin{tikzpicture}[line width=1pt]
\node [rectangle,draw]   (1) at (-5, 4) {1};
\node [rectangle,draw]   (2) at (-3, 4){2};
\node [rectangle,draw]   (3) at (-1, 4){3};
\node [rectangle,draw]   (4) at (-5, 3){4};
\node [rectangle,draw]   (5) at (-3, 3){5};
\node[draw=none] (d1t) at (-10,4) {A};
\node[draw=none] (d1t) at (-10,3) {B };
\node[draw=none] (d1t) at (-10,2) {C};
	\draw[->,color=red]  (1) to (2);
	\draw[->,color=red]  (2) to (3);
	\draw[->,color=blue]  (1) to (5);
	\draw[->,color=green]  (4) to (5);
	\draw[red] (-9,4) -- (-7,4);
	\draw[blue] (-9,3) -- (-7,3);
	\draw[green] (-9,2) -- (-7,2);
\end{tikzpicture}
\caption{System $T_3$} 
\label{fig:T3}
\end{figure}

\end{document}